\documentclass[12pt,twoside,reqno]{amsart}
\usepackage[colorlinks=false,citecolor=blue]{hyperref}
\usepackage{mathptmx, amsmath, amssymb, amsfonts, amsthm, mathptmx, enumerate, color,mathrsfs}
\usepackage{graphicx}

\usepackage{epstopdf}
\theoremstyle{definition}
\newtheorem{definition}{Definition}[section]
\newtheorem{lemma}[definition]{Lemma}
\newtheorem{theorem}[definition]{Theorem}
\newtheorem{corollary}[definition]{Corollary}
\newtheorem{remark}[definition]{Remark}

\newtheorem{proposition}[definition]{Proposition}

\newcommand{\spe}{{\rm \hskip 0.1em sp\hskip -0.2em}~}
\newcommand{\tr}{{\rm Tr\hskip -0.2em}~}
\newcommand{\vertiii}[1]{{\left\vert\kern-0.25ex\left\vert\kern-0.25ex\left\vert #1 
    \right\vert\kern-0.25ex\right\vert\kern-0.25ex\right\vert}}

\def\a{\alpha}
\def\b{\beta}

\def\|{\parallel}
\def\<{\left \langle}
\def\>{\right \rangle}

\def\s{\sharp}

\newcommand{\NORM}[1]{\left|\!\left|{#1}\right|\!\right|}

\begin{document}

%\title{Some inequalities for spectral geometric mean with applications}
%\date{}
%\maketitle
%\begin{flushright}
%S.Furuichi  (2024.04.19)
%\end{flushright}

% M A K E T I T L E %%%

\title[Some inequalities for spectral geometric mean with applications]{Some inequalities for spectral geometric mean \\ with applications}

\author[S. Furuichi]{Shigeru Furuichi$^1$}
\address{$^1$ Department of Information Science, College of Humanitis and Science, Nihon University, Setagaya-ku, Tokyo, 156-8850, Japan.
\vskip2pt
\,\,\, Department of Mathematics, Saveetha School of Engineering, SIMATS, Thandalam, Chennai -- 602105, Tamilnadu, India.}
\email{furuichi.shigeru@nihon-u.ac.jp}

%\author{Yuki Seo}
\author[Y. Seo]{Yuki Seo$^2$}
\address{$^2$ Department of Mathematics Education, Osaka Kyoiku University, Asahigaoka, Kashiwara, Osaka 582-8582, Japan}
\email{yukis@cc.osaka-kyoiku.ac.jp}

\setcounter{page}{1}


%\maketitle

%----------------------------------------------------------%

\begin{abstract}
Recently, the spectral geometric mean has been studied by some papers.
 In this paper, we firstly estimate the H\"{o}lder type inequality of the spectral geometric mean of positive invertible operators on the Hilbert space for all real order in terms of the generalized Kantorovich constant and show the relation between the weighted geometric mean and the spectral geometric one under the usual operator order. Moreover, we show their operator norm version.  Next, in the matrix case, we show the log-majorization for the spectral geometric mean and their applications. Among others, we show the order relation among three quantum Tsallis relative entropies. Finally we give a new lower bound of the Tsallis relative entropy.
\end{abstract}

\maketitle

%---------------------------------------
\section{Introduction}
%---------------------------------------

Let $\mathcal{B}(\mathscr{H})$ be the C$^*$-algebra of all bounded linear operators on a Hilbert space $\mathscr{H}$. $A\in \mathcal{B}(\mathscr{H})$ is said to be positive if $\< Ax,x\> \geq 0$ for all $x\in \mathscr{H}$, and we write $A\geq 0$. If in addition $A$ is invertible, we write $A>0$. The usual operator order $A\geq B$ means $A-B \geq 0$. In addition, for $\a, \b \in {\Bbb R}$, we write $\a \leq A\leq \b$ if $\a I \leq A \leq \b I$ where $I$ is the identity operator. Let $\NORM{A}$ be the operator norm of $A\in \mathcal{B}(\mathscr{H})$. \par

For $A,B > 0$, the weighted geometric mean is defined by 
\begin{equation*}\label{eq_def_geo01}
A\sharp_t B=A^{\frac{1}{2}}\left(A^{-\frac{1}{2}}BA^{-\frac{1}{2}}\right)^t A^{\frac{1}{2}}, \,\, (0\leq t\leq 1),
\end{equation*}
also see \cite{KA}. The weighted spectral geometric mean was also introduced in \cite{LL2007} 
 \begin{equation*}\label{eq_def_Spec_geo01}
 A{{\spe}_{t}}B={{\left( {{A}^{-1}}\sharp B \right)}^{t}}A{{\left( {{A}^{-1}}\sharp B \right)}^{t}}, \quad (0\leq t\leq 1).
 \end{equation*}
For simplicity,  we use the symbols $\sharp$ and $\spe$ instead of $\sharp_{\frac{1}{2}}$ and $\spe_{\frac{1}{2}}$, respectively. Recently, the spectral geometric mean has been studied in relation to other existing means. See \cite{DF2024,GT2022,GK2024} for example.

The generalized Kantorovich constant for $0<m<M$ is defined by
\begin{equation*}%\label{def_K00}
K\left( m,M,t \right)=\frac{\left(m{{M}^{t}}-M{{m}^{t}}\right)}{\left( t-1 \right)\left( M-m \right)}{{\left( \frac{t-1}{t}\frac{{{M}^{t}}-{{m}^{t}}}{m{{M}^{t}}-M{{m}^{t}}} \right)}^{t}}, \quad (t\in \mathbb{R}).
\end{equation*}
It is known that the original Kantorovich constant \cite{K1948} is given by 
$K(m,M,2)=K(m,M,-1)=\dfrac{(M+m)^2}{4Mm}$. We also use
\begin{equation}\label{def_K01}
	K(h,p):=\frac{(h^p-h)}{(p-1)(h-1)}\left(\frac{p-1}{p}\frac{h^p-1}{h^p-h}\right)^p,\quad h>0,\,\,\,\,p\in\mathbb{R}.
\end{equation}
We see that $K(m,M,p)=K(h,p)$ for $h=\frac{M}{m}$. We use the same symbol $K$. For detail, see \cite[Chapter 2]{FMPS2}.

 The following H\"{o}lder type inequality and its reverse for the weighted geometric mean holds for positive operators $A$ and $B$ such that $0<m \leq A,B \leq M$ for some scalars $0<m<M$: For $t\in [0,1]$
\begin{equation} \label{eq:H1}
K\left(\frac{m}{M}, \frac{M}{m}, t\right) \< Ax,x\>^{1-t} \< Bx,x\>^t \leq \< A \s_{t} B x,x\> \leq \< Ax,x\>^{1-t} \< Bx,x\>^t
\end{equation}
for every unit vector $x\in \mathscr{H}$. Unfortunately, the spectral geometric mean version corresponding to the second inequality of \eqref{eq:H1} does not hold in general. Then we have shown the following results in \cite{MFS2023}.
 \begin{theorem}\label{previous_theorem}
Let $A,B\in \mathcal B\left( \mathcal H \right)$ be two positive invertible operators such that $0<m\le A,B\le M$ for some scalars $0<m<M$, and let $0\le t\le 1$. Then for any unit vector $x\in \mathcal H$
\[\frac{K^2\left(\sqrt{\frac{m}{M}},\sqrt{\frac{M}{m}},t\right)}{K(m,M,2)}{{\left\langle Ax,x \right\rangle }^{1-t}}{{\left\langle Bx,x \right\rangle }^{t}}\le \left\langle A{{\spe}_{t}}Bx,x \right\rangle \le K(m^{1+t},M^{1+t},2){{\left\langle Ax,x \right\rangle }^{1-t}}{{\left\langle Bx,x \right\rangle }^{t}}.\]
\end{theorem}

 Though it is known in \cite{GT2022} for the case of matrices that 
\begin{equation} \label{eq:ongs}
\NORM{A\s_{t}B} \leq \NORM{A{{\spe}_{t}}B} \qquad \mbox{for all $t\in [0,1]$},
\end{equation}
there is no usual L\"owner order relation between the weighted geometric mean and the spectral geometric one. Thus, applying the bound of the generalized Kantorovich constant with Theorem \ref{previous_theorem}, we also obtained the following result in \cite{MFS2023}.
\begin{proposition}\label{previous_proposition}
Let $A,B\in \mathcal B\left( \mathcal H \right)$ be two positive invertible operators such that $0<m\le A,B\le M$ for some scalars $0<m<M$, and let $0\le t\le 1$. Then
\begin{equation*}%\label{previous_proposition_ineq01}
\Gamma(m,M,t)^{-1}A{{\sharp }_{t}}B\le \eta(m,M,t)^{-1}A{{\sharp }_{t}}B \le A{{\spe}_{t}}B\le \eta(m,M,t)A{{\sharp }_{t}}B\le \Gamma(m,M,t)A{{\sharp }_{t}}B,
\end{equation*}
where 
$$
\eta(m,M,t) :=\dfrac{K(m,M,2)}{{{K}^{2}}\left( \sqrt{\frac{m}{M}},\sqrt{\frac{M}{m}},t \right)},\quad \text{and}\quad \Gamma(m,M,t):=\dfrac{K(m^{1+t},M^{1+t},2)}{K\left( \frac{m}{M},\frac{M}{m},t \right)}.
$$
\end{proposition}

 Moreover, to investigate the properties of the spectral geometric means, we would like to discuss the above in the case of $t\not\in [0,1]$. For this purpose, we adopt the following symbols:
\begin{equation*}%\label{eq_def_geo02}
A\hat\sharp_t B=A^{\frac{1}{2}}\left(A^{-\frac{1}{2}}BA^{-\frac{1}{2}}\right)^t A^{\frac{1}{2}}, \quad (t\in\mathbb{R})
\end{equation*}
and
\begin{equation*}%\label{eq_def_Spec_geo02}
A{{\hat \spe}_{t}}B={{\left( {{A}^{-1}}\sharp B \right)}^{t}}A{{\left( {{A}^{-1}}\sharp B \right)}^{t}}, \quad (t\in\mathbb{R}).
\end{equation*}
It is known \cite[Proposition 4.1]{LL2007} that $A\hat \spe_tB$ has a unique positive definite solution $X$ for the equation: 
\begin{equation} \label{eq:sgm}
{{\left( {{A}^{-1}}\sharp B \right)}^{t}}={{A}^{-1}}\sharp X,\quad (t\in\mathbb{R}).
\end{equation}

 In this paper, we firstly estimate the H\"{o}lder type inequality of the spectral geometric mean for all real order in terms of the generalized Kantorovich constant and show the relation between the weighted geometric mean and the weighted spectral geometric one under the usual operator order. Moreover, we show their operator norm version.  Next, in the matrix case, we show the log-majorization for the spectral geometric mean and their applications. Among others, we show the order relation among three quantum Tsallis relative entropies.\par
\medskip

%----------------------------------------------
\section{Complementary inequalities for spectral geometric mean with geometric mean}
%----------------------------------------------

To state one of our main results, we prepare the following lemma:
\begin{lemma}\label{lemma2.1}
Let $p \notin (0,1)$, $0<m\le M$, $A,B >0$, and $x\in \mathscr{H}$ be a unit vector.  %Then for every unit vector $x\in \mathscr{H}$
\begin{itemize}
\item[(i)] For $mA\le B \le MA$, 
$$
\langle Ax,x\rangle^{1-p}\langle Bx,x\rangle^p\le \langle A\hat\sharp_p Bx,x\rangle \le K(m,M,p)\langle Ax,x\rangle^{1-p}\langle Bx,x\rangle^p.
$$
\item[(ii)] For $m\le A \le M$, 
$$
\langle Ax,x\rangle^p\le \langle A^px,x\rangle \le K(m,M,p)\langle Ax,x\rangle^p.
$$ 
\item[(iii)] $\langle Ax,x\rangle^{-1}\le \langle A^{-1}x,x\rangle \le K(m,M,2)\langle Ax,x\rangle^{-1}.$
\item[(iv)] For $m\le A\le M$ and $n\le B\le N$, 
$$
\langle A\sharp Bx,x\rangle \le \sqrt{\langle Ax,x\rangle\langle Bx,x\rangle}\le K^{1/2}\left(\sqrt{mn},\sqrt{MN},2\right)\langle A\sharp Bx,x\rangle.
$$
\end{itemize}
\end{lemma}
\begin{proof}
(i) can be proven by the similar way to \cite[Lemma 8]{BLMS2009}. Putting $A:=I$ and $B:=A$ in (i), we have (ii). (iii) is also proven by taking $p:=-1$ in (ii) with $K(m,M,-1)=K(m,M,2)$. The first inequality in (iv) is a special case of the Ando inequality \cite{A1979} and the second inequality is proven from \cite[Theorem 4]{L2009}.
\end{proof}

\begin{theorem} \label{thm-2}
Let $A,B\in \mathcal B\left( \mathcal H \right)$  be two positive invertible operators such that $0<m\le A,B\le M$ for some scalars $0<m<M$, and let $t\in\mathbb{R}$. 
\begin{itemize}
\item[(i)] For $t\ge 1$ and any unit vector $x\in \mathcal H$,
\begin{align}\label{thm2.1_eq01}
K^{-t}(m,M,2)& \langle Ax,x\rangle^{1-t}\langle Bx,x\rangle^t \le \langle A\hat \spe_t Bx,x\rangle \notag \\
& \le K(m^{1+t},M^{1+t},2)K^2(m,M,t) K^{t-1}(m,M,2) \langle Ax,x\rangle^{1-t}\langle Bx,x\rangle^t.
\end{align}
\item[(ii)] For $t\le 0$ and any unit vector $x\in \mathcal H$,
\begin{align}\label{thm2.1_eq02}
K^{t-1}(m,M,2) & \langle Ax,x\rangle^{1-t}\langle Bx,x\rangle^t \le \langle A\hat \spe_t Bx,x\rangle \notag \\
& \le K(m^{t-1},M^{t-1},2)K^{2}(m,M,t) K^{-t}(m,M,2) \langle Ax,x\rangle^{1-t}\langle Bx,x\rangle^t.
\end{align}
\end{itemize}
\end{theorem}

\begin{proof}
Firstly, for $m\le A,B\le M$, we note that $\dfrac{1}{M}\le A^{-1}\le \dfrac{1}{m}$,  $\sqrt{\dfrac{m}{M}}\le A^{-1}\sharp B\le \sqrt{\dfrac{M}{m}}$ and 
\[
\dfrac{m^{1+t}}{M^t}\le  A\hat \spe_t B \le \dfrac{M^{1+t}}{m^t} \quad \mbox{for $t\geq 1$} \quad \mbox{and} \quad \dfrac{M^{t}}{m^{t-1}}\le  A\hat \spe_t B \le \dfrac{m^t}{M^{t-1}} \quad \mbox{for $t\leq 0$.}
\]

Let $X:=A\hat \spe_t B$ for $t\notin (0,1)$.
By the first inequalities in Lemma \ref{lemma2.1} (iv) and (ii), it follows from \eqref{eq:sgm} that for  $t\notin (0,1)$
\begin{equation}\label{thm2.1_eq03}
\sqrt{\langle A^{-1}x,x\rangle\langle Xx,x\rangle}  \ge \langle A^{-1}\sharp Xx,x\rangle
=\langle \left(A^{-1}\sharp B\right)^tx,x\rangle \ge \langle A^{-1}\sharp Bx,x\rangle^t. 
\end{equation}
By the second inequality in Lemma \ref{lemma2.1} (iv), the right hand side in \eqref{thm2.1_eq03} is bounded from below for $t\ge 1$,
$$
\langle A^{-1}\sharp Bx,x\rangle^t \ge K^{-t/2}\left(\sqrt{\frac{m}{M}},\sqrt{\frac{M}{m}},2\right)\langle A^{-1}x,x\rangle^{t/2}\langle Bx,x\rangle^{t/2}=K^{-t/2}\left(m,M,2\right)\langle A^{-1}x,x\rangle^{t/2}\langle Bx,x\rangle^{t/2}.
$$
This with \eqref{thm2.1_eq03} gives
$$
\langle A\hat \spe_tBx,x\rangle \ge K^{-t}\left(m,M,2\right)\langle A^{-1}x,x\rangle^{t-1}\langle Bx,x\rangle^{t}\ge K^{-t}\left(m,M,2\right)\langle Ax,x\rangle^{1-t}\langle Bx,x\rangle^{t}
$$
by the first inequality in Lemma \ref{lemma2.1} (iii). Thus we have the first inequality in \eqref{thm2.1_eq01}.\par
By the first inequality in Lemma \ref{lemma2.1} (iv), the right hand side in \eqref{thm2.1_eq03} is bounded from below for $t\le 0$,
$$
\langle A^{-1}\sharp Bx,x\rangle^t \ge \langle A^{-1}x,x\rangle^{t/2}\langle Bx,x\rangle^{t/2}.
$$
This with \eqref{thm2.1_eq03} gives
$$
\langle A\hat \spe_tBx,x\rangle \ge  \langle A^{-1}x,x\rangle^{t-1}\langle Bx,x\rangle^{t}\ge K^{t-1}(m,M,2) \langle Ax,x\rangle^{1-t}\langle Bx,x\rangle^{t}
$$
by the second inequality in Lemma \ref{lemma2.1} (iii).
Thus we have the first inequality in \eqref{thm2.1_eq02}.\par

By the second inequalities in Lemma \ref{lemma2.1} (iv), (ii) and the first inequality in in Lemma \ref{lemma2.1} (iv), it follows that for $t\geq 1$
\begin{align}
\sqrt{\langle A^{-1}x,x\rangle\langle Xx,x\rangle} & \le K^{1/2}\left(\left(\frac{m}{M}\right)^{\frac{1+t}{2}},\left(\frac{M}{m}\right)^{\frac{1+t}{2}},2\right)\langle A^{-1}\sharp Xx,x\rangle \nonumber \\
&=K^{1/2}\left(m^{1+t},M^{1+t},2\right)\langle \left(A^{-1}\sharp B\right)^tx,x\rangle  \nonumber \\
&\le K^{1/2}\left(m^{1+t},M^{1+t},2\right)K\left(\sqrt{\frac{m}{M}},\sqrt{\frac{M}{m}},t\right)\langle A^{-1}\sharp Bx,x\rangle^t \nonumber \\
&=K^{1/2}\left(m^{1+t},M^{1+t},2\right)K(m,M,t)\langle A^{-1}\sharp Bx,x\rangle^t\nonumber \\
&\le K^{1/2}\left(m^{1+t},M^{1+t},2\right)K(m,M,t)\langle A^{-1}x,x\rangle^{t/2}\langle Bx,x\rangle^{t/2}.
\label{thm2.1_eq04}
\end{align}
Hence it follows from \eqref{thm2.1_eq04} that
\begin{align*}
\langle A\hat \spe_tBx,x\rangle &\le K\left(m^{1+t},M^{1+t},2\right)K^2(m,M,t)\langle A^{-1}x,x\rangle^{t-1}\langle Bx,x\rangle^t \\
&\le  K\left(m^{1+t},M^{1+t},2\right)K^2(m,M,t) K^{t-1}(m,M,2) \langle Ax,x\rangle^{1-t}\langle Bx,x\rangle^t \quad \mbox{for $t\geq 1$}
\end{align*}
by the second inequality in Lemma \ref{lemma2.1} (iii). Thus we have the second inequality in \eqref{thm2.1_eq01}. \par
For $t \leq 0$, by the second inequalities in Lemma \ref{lemma2.1} (iv), (ii) and (iv), we have 
\begin{align}
\sqrt{\langle A^{-1}x,x\rangle\langle Xx,x\rangle} & \le K^{1/2}\left(\left(\frac{M}{m}\right)^{\frac{t-1}{2}},\left(\frac{m}{M}\right)^{\frac{t-1}{2}},2\right)\langle A^{-1}\sharp Xx,x\rangle \nonumber \\
&=K^{1/2}\left(M^{t-1},m^{t-1},2\right)\langle \left(A^{-1}\sharp B\right)^tx,x\rangle  \nonumber \\
&\le K^{1/2}\left(M^{t-1},m^{t-1},2\right)K(m,M,t)\langle A^{-1}\sharp Bx,x\rangle^t \nonumber \\
& \leq K^{1/2}\left(M^{t-1},m^{t-1},2\right)K(m,M,t) K^{-\frac{t}{2}}(m,M,2) \langle A^{-1}x,x\rangle^{t/2}\langle Bx,x\rangle^{t/2}. \label{thm2.1_eq04-2}
\end{align}

%\begin{align*}
%& K^{1/2}\left(m^{1+t},M^{1+t},2\right)K(m,M,2)\langle A^{-1}\sharp Bx,x\rangle^t \\
%&\le K^{1/2}\left(m^{1+t},M^{1+t},2\right)K(m,M,2)K^{-t/2}\left(\sqrt{\frac{m}{M}},\sqrt{\frac{M}{m}},2\right)\langle A^{-1}x,x\rangle^{t/2}\langle Bx,x\rangle^{t/2} \\
%&=K^{1/2}\left(m^{1+t},M^{1+t},2\right)K^{1-t/2}(m,M,2)\langle A^{-1}x,x\rangle^{t/2}\langle Bx,x\rangle^{t/2}.
%\end{align*}
%This with \eqref{thm2.1_eq04-2} gives
Hence it follows from \eqref{thm2.1_eq04-2} that
\begin{align*}
\langle A\hat \spe_tBx,x\rangle &\le K\left(M^{t-1},m^{t-1},2\right) K^2(m,M,t)K^{-t}(m,M,2)\langle A^{-1}x,x\rangle^{t-1}\langle Bx,x\rangle^{t} \\
&\le K\left(M^{t-1},m^{t-1},2\right) K^2(m,M,t)K^{-t}(m,M,2)\langle Ax,x\rangle^{1-t}\langle Bx,x\rangle^{t},
\end{align*}
by the first inequality in Lemma \ref{lemma2.1} (iii).
Thus we have the second inequality in \eqref{thm2.1_eq02}.
\end{proof}

By virtue of Theorem~\ref{thm-2}, we have the following usual operator order relation between the weighted geometric mean and the weighted spectral geometric one in terms of the generalized Kantorovich constant:

\begin{corollary}
Let $A,B\in \mathcal B\left( \mathcal H \right)$  be two positive invertible operators such that $0<m\le A,B\le M$ for some scalars $0<m<M$, and let $t\in\mathbb{R}$. 
\begin{itemize}
\item[(i)] For $t \ge 1$, 
\[
\chi_+(m,M,t)^{-1} A\hat\sharp_t B \le A\hat \spe_t B \le \chi_+(m,M,t) A\hat\sharp_t B,
\]
\item[(ii)] for $t \le 0$,
\[ 
\chi_-(m,M,t)^{-1} A\hat\sharp_t B \le A\hat \spe_t B \le \chi_-(m,M,t) A\hat\sharp_t B,
\] 
\end{itemize}
where $\chi_{\pm}(m,M,t):=\min\left\{\eta_{\pm}(m,M,t),\Gamma_{\pm}(m,M,t)\right\}$, and
\begin{align*}
&\eta_{+}(m,M,t):=K^t(m,M,2)K\left(\frac{m}{M},\frac{M}{m},t\right),\\
& \eta_{-}(m,M,t):=K^{1-t}(m,M,2)K\left(\frac{m}{M},\frac{M}{m},t\right),\\
& \Gamma_{+}(m,M,t):=K\left(m^{1+t},M^{1+t},2\right)K^{2}(m,M,t)K^{t-1}(m,M,2),\\
& \Gamma_{-}(m,M,t):=K\left(m^{t-1},M^{t-1},2\right)K^{2}(m,M,t) K^{-t}(m,M,2).
\end{align*}
\end{corollary}
\begin{proof}
Note that $\dfrac{m}{M}A\le B \le \dfrac{M}{m}A$ from $m\le A,B\le M$. By the second inequality in Lemma \ref{lemma2.1} (i) with the first inequality in \eqref{thm2.1_eq01}, we have
$$
A\hat\sharp_tB\le \eta_+(m,M,t)A\hat \spe_tB.
$$
 Replacing $A$ and $B$ by $A^{-1}$ and $B^{-1}$ respectively, we have 
$$
A\hat \spe_tB\le \eta_+(m,M,t)A\hat\sharp_tB
$$
since $\eta_+\left(\frac{1}{M},\frac{1}{m},t\right)=\eta_{+}(m,M,t)$, $\left(A^{-1}\hat\sharp_tB^{-1}\right)^{-1}=A\hat\sharp_tB$ and $\left(A^{-1}\hat \spe_tB^{-1}\right)^{-1}=A\hat \spe_tB$. By the first inequality in Lemma \ref{lemma2.1} (i) with the second inequality in \eqref{thm2.1_eq01}, we have
$$
A\hat \spe_tB\le \Gamma_{+}(m,M,t) A\hat\sharp_t B.
$$
Thus we obtained the second inequality in (i). By the self-duality of two geometric means and $\chi_+(M^{-1},m^{-1},t)=\chi_+(m,M,t)$ by the property of the Kantorovich constant, we have the first inequality in (i).

By the second inequality in Lemma \ref{lemma2.1} (i) with the first inequality in \eqref{thm2.1_eq02}, we have
$$
A\hat\sharp_tB \le K\left(\frac{m}{M},\frac{M}{m},t\right)K^{1-t}(m,M,2) A\hat \spe_tB.
$$
Replacing $A$ and $B$ by $A^{-1}$ and $B^{-1}$ respectively, we have 
$$
A\hat \spe_tB \le K\left(\frac{m}{M},\frac{M}{m},t\right)K^{1-t}(m,M,2) A\hat\sharp_tB.
$$
By the first inequality in Lemma \ref{lemma2.1} (i) with the second inequality in \eqref{thm2.1_eq02}, we have
$$
A\hat \spe_tB\le \Gamma_{-}(m,M,t) A\hat\sharp_t B.
$$
Thus we obtain the inequalities in (ii), by the self-duality of two geometric means $\hat\spe_t$ and $\hat\sharp_t$, and $\chi_-(M^{-1},m^{-1},t)=\chi_-(m,M,t)$ .
\end{proof}

\begin{remark}
\begin{itemize}
It is quite natural to compare $\eta_{\pm}(m,M,t)$ and $\Gamma_{\pm}(m,M,t)$.
We firstly show numerical examples for them.
\item[(i)] For $t\ge 1$ and $0<m\le M$, we have the following numerical examples.
\begin{align*}
& \eta_+(3,4,2)\simeq 1.13075 < 1.27451 \simeq \Gamma_+(3,4,2)\\
& \eta_+(3,4,5)\simeq 2.36637 < 3.17105 \simeq \Gamma_+(3,4,5).
\end{align*}
\item[(ii)]  For $t\le 0$ and $0<m\le M$, we have the following numerical examples.
\begin{align*}
& \eta_-(3,4,-2)\simeq 1.35329 < 1.41213 \simeq \Gamma_-(3,4,-2)\\
& \eta_-(3,4,-1.5)\simeq 1.22562 < 1.26455\simeq \Gamma_-(3,4,-1.5).
\end{align*}
\end{itemize}
The above numerical computations support the relation $\eta_{\pm}(m,M,t)\le \Gamma_{\pm}(m,M,t)$.
However it is not easy to show these inequalities.
We note that
\begin{align*}
& \eta_{\pm}(m,M,t)\le \Gamma_{\pm}(m,M,t) \\
\Longleftrightarrow &K(m,M,2)K\left(M/m,m/M,t\right)\le K(m^{t\pm 1},M^{t\pm 1},2)K^2(m,M,t)\\
\Longleftrightarrow &K(x,2)K(x^2,t)\le K(x^{t\pm 1},2)K^2(x,t).
\end{align*}
It is known that $K(x^2,t)\ge K^2(x,t),\,\,(t\le 0,\,\,{\rm or}\,\,t\ge 1)$ in \cite[Corollary 3.2]{MFS2023}. In addition, it is easy to see that $K(x,2)\le K(x^{t+1},2)$ for $t \ge 1$ and  $K(x,2)\le K(x^{t-1},2)$ for $t\le 0$ since the following calculations
\begin{align*}
&\frac{dK(x^{t+1},2)}{dt}=\frac{\left(x^{1+t}-1\right)\left(x^{1+t}+1\right)\log x}{4x^{1+t}} \ge 0,\,\,(x>0,\,\,t\ge 1)\\
&\frac{dK(x^{t-1},2)}{dt}=\frac{\left(x^{t}-x\right)\left(x^{t}+x\right)\log x}{4x^{t+1}} \le 0,\,\,(x>0,\,\,t\le 0)
\end{align*}
imply $K(x^{t+1},2)\ge K(x,2)$ and $K(x^{t-1},2)\ge K(x^{-1},2)=K(x,2)$.
Thus we have to prove the following inequalities directly
\begin{align*}
&K(x,2)K(x^2,t)\le K(x^{t+ 1},2)K^2(x,t),\quad (x>0,\,\,\,t\ge 1),\\
&K(x,2)K(x^2,t)\le K(x^{t- 1},2)K^2(x,t),\quad (x>0,\,\,\,t\le 0).
\end{align*}
We have not proven these inequalities and not found any counter--examples yet.
\end{remark}

%----------------------------------------------
\section{Norm inequalities for spectral geometric mean with geometric mean}
%----------------------------------------------

Next, we show norm inequalities related to two geometric means of positive invertible operators for all $t \in {\Bbb R}$. Before we state that, we need some preliminaries. We recall the following Kantorovich inequality in \cite[Chapter 8]{FMPS1}: Let $A$ and $B$ be positive invertible operators such that $m_1\leq A \leq M_1$ and $m_2 \leq B \leq M_2$ for some scalars $0<m_1<M_1$ and $0<m_2< M_2$. Put $h_1:=\frac{M_1}{m_1}$ and $h_2:=\frac{M_2}{m_2}$. If $A\leq B$, then 
\[
A^p \leq K(h_1,p) B^p \qquad \mbox{for all $p\geq 1$}
\]
and 
\[
A^p \leq K(h_2,p) B^p \qquad \mbox{for all $p\geq 1$},
\] 
where $K(h,p)$ is the generalized Kantorovich constant given in \eqref{def_K01}.\par

\begin{theorem} \label{thm-m}
Let $A$ and $B$ be positive invertible operators such that $m\leq A, B \leq M$ for some scalars $0<m<M$, and $t \in {\Bbb R}$. Put $h:=\frac{M}{m}$. Then
\begin{enumerate}
\item[(i)] For $t \leq -\frac{1}{2}$
\[
K(h,-2t)^{-1} \NORM{A \hat \spe_t B} \leq \NORM{A \hat \s_{t} B} \leq K(h^2,1-t) K(h,2)^{1-t} \NORM{A \hat \spe_t B}.
\]
\item[(ii)] For $-\frac{1}{2}\leq t \leq 0$
\[
\NORM{A \hat \spe_t B} \leq \NORM{A \hat \s_{t} B} \leq K(h^2,1-t) K(h,2)^{1-t} \NORM{A \hat \spe_t B}.
\]
\item[(iii)] For $0\leq t \leq \frac{1}{2}$
\[
K(h,2(1-t))^{-1} K\left(h,\frac{1}{1-t}\right)^{t-1} \NORM{A \spe_t B} \leq \NORM{A \s_{t} B} \leq \NORM{A \spe_t B}.
\]
\item[(iv)] For $\frac{1}{2}\leq t \leq 1$
\[
K(h,2t)^{-1} K\left(h,\frac{1}{t}\right)^{-t} \NORM{A \spe_t B} \leq \NORM{A \s_{t} B} \leq \NORM{A \spe_t B}.
\]
\item[(v)] For $1 \leq t \leq \frac{3}{2}$
\[
\NORM{A \hat \spe_t B} \leq \NORM{A \hat\sharp_t B} \leq K(h^2,t) K(h,2)^{t} \NORM{A \hat \spe_t B}.
\]
\item[(vi)] For $\frac{3}{2} \leq t$
\[
K(h,2(t-1))^{-1} \NORM{A \hat \spe_t B} \leq \NORM{A \hat\sharp_t B} \leq K(h^2,t) K(h,2)^{t} \NORM{A \hat \spe_t B}
\]
\end{enumerate}
\end{theorem}

\begin{proof}
(iii): Suppose that $0\leq t \leq \frac{1}{2}$ and $A \s_{t} B = B \s_{1-t} A \leq I$.  Since
\[
(B^{-1/2}AB^{-1/2})^{1-t} \leq B^{-1}
\]
and the generalized condition number of $B^{-1}$ is $h=\frac{M}{m}$, it follows from $\frac{1}{1-t} \geq 1$ and the Kantorovich inequality  that
\[
B^{-1/2}AB^{-1/2} \leq K\left(h,\frac{1}{1-t}\right) B^{-\frac{1}{1-t}}
\]
and $A\leq K\left(h,\frac{1}{1-t}\right) B^{\frac{t}{t-1}}$. Thus, by the monotonicity of the operator mean, we have
\[
B^{-1} \s A \leq K\left(h,\frac{1}{1-t}\right)^{1/2} B^{-1} \s B^{\frac{t}{t-1}}  = K\left(h,\frac{1}{1-t}\right)^{1/2} B^{\frac{1}{2(t-1)}}.
\]
Since $1\leq 2(1-t) \leq 2$ and $\sqrt{\dfrac{m}{M}} \le B^{-1} \s A\le \sqrt{\dfrac{M}{m}}$ which implies $h:=\dfrac{\sqrt{M/m}}{\sqrt{m/M}}=\dfrac{M}{m}$, it follows from  Kantorovich inequality that
\[
(B^{-1} \s A)^{2(1-t)} \leq K(h, 2(1-t)) K\left(h,\frac{1}{1-t}\right)^{1-t} B^{-1}
\]
and so
\[
B^{1/2} (B^{-1} \s A)^{2(1-t)} B^{1/2} \leq K(h, 2(1-t)) K\left(h,\frac{1}{1-t}\right)^{1-t}.
\]
Hence we have the following implication:
\[
A \s_{t} B \leq I \quad \Longrightarrow \quad  K\left(h,2(1-t)\right)^{-1}K\left(h,\frac{1}{1-t}\right)^{t-1}B^{1/2}\left(B^{-1}\s A\right)^{2(1-t)}B^{1/2}\le I
\]
for all $t \in [0,1/2]$. Since $A \spe_{t} B$ and $A \s_{t} B$ have the same order of homogeneity for $A,B$, we have
\[
K\left(h,2(1-t)\right)^{-1}K\left(h,\frac{1}{1-t}\right)^{t-1}\NORM{B^{1/2}\left(B^{-1}\s A\right)^{2(1-t)}B^{1/2}}\le \NORM{A\s_t B},
%\NORM{A \spe_{t} B} \leq K(h, 2(1-t)) K\left(h,\frac{1}{1-t}\right)^{1-t} \NORM{A \s_{t} B} \qquad \mbox{for all $t \in [0,1/2]$}.
\]
and equivalently,
$$
\NORM{A \spe_t B} \le K\left(h,2(1-t)\right)K\left(h,\frac{1}{1-t}\right)^{1-t}\NORM{A \s_t B}.
$$
(iv): Suppose that $\frac{1}{2}\leq t \leq 1$. Since $0\leq 1-t \leq \frac{1}{2}$, it follows from (iii) that
\[
K(h,2t)^{-1} K\left(h,\frac{1}{t}\right)^{-t} \NORM{B \spe_{1-t} A} \leq \NORM{B \s_{1-t} A}
\]
and so
\[
K(h,2t)^{-1} K\left(h,\frac{1}{t}\right)^{-t} \NORM{A \spe_{t} B} \leq \NORM{A \s_{t} B} \leq \NORM{A \spe_{t} B} \qquad \mbox{for all $t \in [1/2,1]$}.
\]
(v): Suppose that $1\leq t \leq \frac{3}{2}$ and $A \hat\s_{t} B \leq I$. Since $(A^{-1/2}BA^{-1/2})^{t} \leq A^{-1}$ and $0<\frac{1}{t}\leq 1$, we have $A^{-1/2}BA^{-1/2}\leq A^{-\frac{1}{t}}$ and so $B\leq A^{1-\frac{1}{t}}$. Hence we have $B^{-1} \s A \geq A^{\frac{1}{2t}}$. Since $-1\leq 2(1-t)\leq 0$, it follows that
\[
(B^{-1} \s A)^{2(1-t)} \leq  A^{\frac{1}{t}-1}
\]
and so
\[
B^{1/2} (B^{-1} \s A)^{2(1-t)} B^{1/2} \leq B^{1/2}A^{\frac{1}{t}-1}B^{1/2}\leq I.
\]
Hence we have
\[
\NORM{A \hat \spe_{t} B}= \NORM{B \hat \spe_{1-t} A} \leq \NORM{A \hat\s_{t} B} \qquad \mbox{for all $1\leq t \leq \frac{3}{2}$}.
\]
Next, suppose that $t\geq 1$ and $A \hat\spe_{t} B \leq I$. Since $A^{1/2}(A^{-1} \s B)^{2t} A^{1/2} \leq I$ and so $(A^{-1} \s B)^{2t} \leq A^{-1}$, it follows from $0<\frac{1}{2t}\leq 1$ that $A^{-1} \s B \leq A^{-\frac{1}{2t}}$ and so
\[
(A^{1/2}BA^{1/2})^{1/2} \leq A^{1-\frac{1}{2t}}.
\]
Taking the square of both sides, it follows from $m\leq (A^{1/2}BA^{1/2})^{1/2} \leq M$ that
\[
A^{1/2}BA^{1/2} \leq K(h, 2) A^{2-\frac{1}{t}}
\]
and so $A^{-1/2} B A^{-1/2}\leq K(h,2) A^{-\frac{1}{t}}$. Taking the $t$-power of both sides, it follows from $m/M \leq A^{-1/2} B A^{-1/2} \leq M/m$ that 
\[
(A^{-1/2}BA^{-1/2})^{t} \leq K(h^2,t) K(h,2)^{t} A^{-1}
\]
and so $A \hat\s_{t} B \leq K(h^2,t) K(h,2)^{t}$. Hence we have
\[
\NORM{A \hat\s_{t} B} \leq K(h^2,t) K(h,2)^{t} \NORM{A \hat \spe_t B}
\]
for all $t\geq 1$, and we have (v) and the second inequality of (vi).\\
(vi): Suppose that $\frac{3}{2} \leq t$. If $A\hat \s_{t} B \leq I$, then it follows from L\"{o}wner-Heinz inequality that $B\leq A^{1-\frac{1}{t}}$. Since $1\leq 2(t-1)$, it follows from $B \s A^{-1}\leq A^{-\frac{1}{2t}}$ that 
\begin{align*}
B^{\frac{1}{2}}(B^{-1} \s A)^{2(1-t)} B^{\frac{1}{2}} & = B^{\frac{1}{2}}(B \s A^{-1})^{2(t-1)} B^{\frac{1}{2}} \\
& \leq B^{\frac{1}{2}} K(h,2(t-1))A^{\frac{1}{t}-1}B^{\frac{1}{2}} \\
& \leq K(h,2(t-1)).
\end{align*}
Since $A\hat{\spe_t}B =B\hat{\spe_{1-t}}A$ (see \cite[Lemma 1]{K2021} for example),  we have 
\begin{align*}
\NORM{A \hat \spe_{t} B} & = \NORM{B\hat{\spe_{1-t}}A} \\
& = \NORM{(B^{-1} \s A)^{1-t} B (B^{-1} \s A)^{1-t}} \\
& = \NORM{B^{\frac{1}{2}}(B^{-1} \s A)^{2(1-t)} B^{\frac{1}{2}}}\\
& \leq K(h,2(t-1)) \NORM{A \hat \s_{t} B}
\end{align*}
and we have (vi).\par
(i): Suppose that $t \leq -\frac{1}{2}$. Since $1-t \geq \frac{3}{2}$, it follows from (vi) that
\[
K(h,-2t)^{-1}  \NORM{B \hat \spe_{1-t} A} \leq \NORM{B \hat \s_{1-t} A} \leq K(h^2,1-t) K(h,2)^{1-t} \NORM{B \hat \spe_{1-t} A}
\]
and so
\[
K(h,-2t)^{-1}  \NORM{A \hat \spe_{t} B} \leq \NORM{A \hat \s_{t} B} \leq K(h^2,1-t) K(h,2)^{1-t} \NORM{A \hat \spe_{t} B} 
\]
for all $t \in [1/2,1]$.\par
(ii): Suppose that $-\frac{1}{2}\leq t\leq 0$. Since $1\leq 1-t \leq \frac{3}{2}$, it follows from (v) that
\[
 \NORM{B \hat \spe_{1-t} A} \leq \NORM{B \hat \s_{1-t} A} \leq K(h^2,1-t) K(h,2)^{1-t} \NORM{B \hat \spe_{1-t} A}
\]
and so
\[
\NORM{A \hat \spe_{t} B} \leq \NORM{A \hat \s_{t} B} \leq K(h^2,1-t) K(h,2)^{1-t} \NORM{A \hat \spe_{t} B} 
\]
for all $t \in [-\frac{1}{2}, 0]$.\par
Hence the proof of Theorem~\ref{thm-m} is complete.
\end{proof}

\begin{remark}
If $t=0$ in (iii), then $K(h,2(1-t))^{-1} K\left(h,\frac{1}{1-t}\right)^{t-1}=K(h,2)^{-1} \not= 1$ for $h\neq 1\Leftrightarrow m\neq M$. However, if $t=\frac{1}{2}$ in (iii), then we have 
\begin{equation}\label{remark3.2_eq01}
\NORM{A \spe B}\leq K(h,2)^{\frac{1}{2}}\NORM{A \s B}.
\end{equation}
On the other hand, by Proposition~\ref{previous_proposition} we have
\begin{equation}\label{remark3.2_eq02}
\NORM{A \spe B}\leq \frac{K(h,2)}{K^2(h,\frac{1}{2})} \NORM{A \s B}.
\end{equation}
Since we have $K^2\left(h,\frac12\right) \leq K^{1/2}(h,2)$ for $h>0$ by elementary calculations, the inequality \eqref{remark3.2_eq01} is better than the inequality \eqref{remark3.2_eq02}. Indeed, 
$$K^2\left(h,\frac12\right) \leq K^{1/2}(h,2)\Longleftrightarrow \left(\frac{h+1}{2\sqrt{h}}\right)^2\left(\frac{\sqrt{h}+1}{2h^{1/4}}\right)^4\ge 1
$$
which is true for $h>0$.
\end{remark}
\par
\medskip

Now, we want to refine the obtained norm inequality. For example, by \eqref{eq:ongs}, the norm inequality $ 
\NORM{A\s_{t}B} \leq \NORM{A{{\spe}_{t}}B}$ holds for all $t\in [0,1]$. 
On the other hand, it is known in \cite[Theorem 1]{NS} and \cite[Theorem 4]{S2007} that for all $t\in [0,1]$
\begin{equation*} %\label{eq:se22}
\NORM{A\s_{t}B} \leq \NORM{e^{(1-t)\log A + t \log B}} \leq \NORM{A^{\frac{1-t}{2}}B^t A^{\frac{1-t}{2}}} \leq \NORM{A^{1-t} B^t}.
\end{equation*}
We slightly refine the norm inequality \eqref{eq:ongs} in the following. We omit its proof which is the same as that of \cite[Theorem 3.6]{GT2022}.
\begin{theorem}
Let $A,B\in \mathcal B\left( \mathcal H \right)$  be two positive invertible operators. Then for each $t\in [0,1]$
\begin{equation*} %\label{eq:se2s}
\NORM{A\s_{t}B} \leq \NORM{e^{(1-t)\log A + t \log B}} \leq \NORM{A^{\frac{1-t}{2}}B^t A^{\frac{1-t}{2}}} \leq \NORM{A{{\spe}_{t}}B}.
\end{equation*}
\end{theorem}

%\begin{proof}
%If $A{{\spe}_{t}}B \leq I$ for some $t\in [0,1]$, then $C^tAC^t\leq I$ where $C=A^{-1}\s B$, and so $A\leq C^{-2t}$. Also, it follows from the Riccati equation that $B=CAC \leq C^{2-2t}$. Since $t\in [0,1]$, we have $B^t \leq C^{t(2-2t)}$ and
%\[
%A^{\frac{1-t}{2}} B^t A^{\frac{1-t}{2}} \leq A^{\frac{1-t}{2}} C^{t(2-2t)} A^{\frac{1-t}{2}}.
%\]
%Hence, it follows that
%\begin{align*}
%\NORM{A^{\frac{1-t}{2}} B^t A^{\frac{1-t}{2}}} & \leq \NORM{A^{\frac{1-t}{2}} C^{t(2-2t)} A^{\frac{1-t}{2}}} \\
%& = \NORM{C^{t(1-t)} A^{1-t}C^{t(1-t)}} \\
%& \leq \NORM{C^{t(1-t)} C^{-2t(1-t)}C^{t(1-t)}}=1
%\end{align*} 
%and so $\NORM{A^{\frac{1-t}{2}}B^t A^{\frac{1-t}{2}}} \leq \NORM{A{{\spe}_{t}}B}$.
%\end{proof}

Next, we consider the case of $t\not\in [0,1]$:

\begin{theorem}\label{norm_main}
Let $A,B\in \mathcal B\left( \mathcal H \right)$  be two positive invertible operators. Then for each $t\in [-\frac{1}{2},0] \cup [1,\frac{3}{2}]$
\begin{equation} \label{eq:s2s}
\NORM{A \hat \spe_t B} \leq \NORM{A^{\frac{1-t}{2}} B^t A^{\frac{1-t}{2}}} \leq \NORM{A \hat\s_t B}.
\end{equation}
\end{theorem}

\begin{proof}
Let $-\frac{1}{2}\leq t\leq 0$. By a similar way in the proof of Theorem~\ref{theorem3.1}, it follows that for two positive invertible operators $A$ and $B$
\[
A^{\frac{1-t}{2}} B^t A^{\frac{1-t}{2}} \leq I \quad \mbox{implies} \quad A \hat \spe_t B\leq I
\]
and so $\NORM{A \hat s_t B} \leq \NORM{A^{\frac{1-t}{2}} B^t A^{\frac{1-t}{2}}}$ for $t\in [-\frac{1}{2},0]$. Also, it follows from \cite[Theorem 2.1]{BLP2005} that
\[
\NORM{A^{\frac{1-t}{2}} B^t A^{\frac{1-t}{2}}} \leq \NORM{A \hat\s_t B} \quad \mbox{for $t\in [-1,0]$}.
\]
Hence we have the norm inequality \eqref{eq:s2s} for $t\in [-\frac{1}{2},0]$.\par
Let $1 \leq t \leq \frac{3}{2}$. Since $-\frac{1}{2}\leq 1-t \leq 0$, it follows from the discuss above that the norm inequality \eqref{eq:s2s} holds for $1 \leq t \leq \frac{3}{2}$.
\end{proof}

\section{Log--majorizations for spectral geometric mean and their applications}
In this section and beyond, we will deal with matrices rather than operators.
The study on quantification for quantum coherence has been initiated in \cite{BCP2014} and 
the Tsallis relative entropy and the Tsallis relative operator entropy have been considered to be possible measures in \cite{GJBBHF2020, R2016, ZY2018}. It may be interesting to study the relations among the Tsallis relative entropy, the Tsallis relative operator entropy  and the Tsallis relative operator entropy  due to the spectral geometric mean in this section. See the paper \cite{KL} which firstly studied the Tsallis relative entropy with spectral geometric mean.

From \cite[Corollary 2.4]{AH1994} and \cite[Theorem 3.6]{GT2022}, we have
$$
\tr\,A\sharp_t B\le \tr\, \exp\left((1-t)\log A+t \log B\right) \le \tr\,A^{1-t}B^t\le \tr\, A\spe_tB,\,\,\,(0\le t\le 1)
$$
for positive definite matrices$A$ and $B$. In this section and later, we treat the {\it matrix} case
Thus we have
\begin{equation}\label{sec3_eq01}
-\tr\,T^{\spe}_t(A|B)\le D_t(A|B) \le -\tr\, T_t(A|B),\,\,\,(0<t\le 1)
\end{equation}
where 
$D_{t}(A|B):=\tr\,\left[\dfrac{A-A^{1-t}B^t}{t}\right]$ is the Tsallis relative entropy
and 
$$T_t(A|B):=\frac{A\hat \sharp_t B-A}{t}=A^{1/2}\ln_t \left(A^{-1/2}BA^{-1/2}\right)A^{1/2},\,\,(t\neq 0)$$
is the Tsallis relative operator entropy with the $t$--logarithmic function defined by $\ln_t x:=\dfrac{x^t-1}{t}$ for $x>0$ and $t \neq 0$.  See \cite{FYK2004,S2019,FS2021} for example. In addition, we define the Tsallis relative operator entropy due to the spectral geometric mean by 
$$T^{\spe}_t(A|B):=\frac{A\hat\spe_tB-A}{t},\,\,(t\neq 0).$$
Previously many results on $D_t(A|B)$ and $T_t(A|B)$ have been studied. See \cite[Chapter 7]{FM2020book} and references therein. It may be of interest to study the mathematical properties on $T^{\spe}_t(A|B)$.

For two positive semidefinite matrices $X,Y$, the $\log$--majorization $X \prec_{\log} Y$ means that
$$\prod\limits_{j=1}^k\lambda_j(X)\le \prod\limits_{j=1}^k\lambda_j(Y),\,\,(k=1,\cdots,n-1)$$ and $\det X=\det Y$, where $\lambda_1(X)\ge \cdots\ge \lambda_n(X)$ are the eigenvalues of $X$ in decreasing order. 

It is known the  following relations between $A\hat\sharp_rB$ and $A^{\frac{1-r}{2}}B^{r}A^{\frac{1-r}{2}}$. See \cite[Theorem 1.1]{H2019} for (i)--(iii), \cite[Theorem 2.1]{BLP2005} for (iv) and \cite[Theorem 3.1]{MF2009} for (v), for example. 
\begin{proposition}\label{prop_known_fact}
Let $r\in\mathbb{R}$.  For positive definite matrices $A$ and $B$, 
\begin{itemize}
\item[(i)] For $0 \le r \le 1$, $A\sharp_r B \prec_{\log} A^{\frac{1-r}{2}}B^{r}A^{\frac{1-r}{2}}$.
\item[(ii)] For $1\le r \le 2$, $A^{\frac{1-r}{2}}B^{r}A^{\frac{1-r}{2}}\prec_{\log} A\hat\sharp_r B$.
\item[(iii)] For $r \ge 2$, $A\hat\sharp_r B \prec_{\log} A^{\frac{1-r}{2}}B^{r}A^{\frac{1-r}{2}}$.
\item[(iv)] For $-1\le r \le 0$, $A^{\frac{1-r}{2}}B^{r}A^{\frac{1-r}{2}}\prec_{\log} A\hat\sharp_r B$.
\item[(v)] For $r \le -1$, $A\hat\sharp_r B \prec_{\log} A^{\frac{1-r}{2}}B^{r}A^{\frac{1-r}{2}}.$
\end{itemize}
\end{proposition}
Indeed, if we take $p:=1$, $q:=-r$ and $B:=B^{-1}$ in \cite[Theorem 2.1]{BLP2005}, then we obtain (iv). Also, if we take $s:=1$, $t:=-r$ and $B:=B^{-1}$ in  \cite[Theorem 3.1]{MF2009} , then we obtain (v).

In addition, it is known the following relation:
\begin{equation*}\label{eq_Gan}
A^{\frac{1-t}{2}}B^tA^{\frac{1-t}{2}} \prec_{\log} A\spe_t B,\quad (0\le t \le 1)
\end{equation*}
from \cite[Theorem 3.6]{GT2022} as a special case. In this paper, we consider the case $t\notin (0,1)$. We also consider an application of the relation among the Tsallis relative entropy, the Tsallis relative operator entropy and the Tsallis relative operator entropy due to the spectral geometric mean. 

In order to show the following result, we prepare the known facts. 
%See \cite{M1970} for Lemma \ref{lemma3.1} (ii).
\begin{lemma}\label{lemma3.1}
Let $X$ and $Y$ be positive definite matrices. %Then
\begin{itemize}
\item[(i)] If $X \le Y$ and $-1\le \alpha \le 0$, then $X^{\alpha} \ge Y^{\alpha}$.
\item[(ii)] If $Z$ is a positive definite matrix and $XZX\le YZY$, then  $X \le Y$.
%\item[(iii)] If $s\ge 0$ and $X\le Y$, then we have $\lambda_1(X^s)\le \lambda_1(Y^s)$.
\end{itemize}
\end{lemma} 

\begin{proof}
(i) can be proven by the facts $X\le Y \Longrightarrow X^{-1}\ge Y^{-1}$ and the L\"owner--Heinz inequality 	$X\le Y  \Longrightarrow X^s\le Y^s$ for $0\le s \le 1$. 

(ii) From the assumptions, we have $Z^{1/2}XZXZ^{1/2}\le Z^{1/2}YZYZ^{1/2}$ so that $\left(Z^{1/2}XZ^{1/2}\right)^2\le \left(Z^{1/2}YZ^{1/2}\right)^2$. By the L\"owner--Heinz inequality, we have
 $Z^{1/2}XZ^{1/2}\le Z^{1/2}YZ^{1/2}$ which implies $X\le Y$.
%Set $W:=YZY-XZX$. Since $W \ge 0$ for a certain $Z> 0$, we have $Z-Y^{-1}XZXY^{-1}=Y^{-1}WY^{-1}\ge 0$. Thus we have $\rho (Y^{-1}X)\le 1$, where $\rho(A):=\max_{i}|\lambda_i|$ represents the spectral radius of $A$, because we can show that if there exists $S> 0$ such that $S-A^*SA\ge 0$, then $\rho(A)\le 1$. 
%
%Indeed, let $x$ is an eigenvector corresponding to $\lambda$ for a matrix $A$, that is, $Ax =\lambda x,\,\,x \neq 0$. 
%Then we have $0\le \langle (S-A^*SA)x,x\rangle=(1-|\lambda|^2)\langle Sx,x\rangle$, which implies $1-|\lambda|^2\ge 0$. Thus we obtain $\rho(A)\le 1$, since $\lambda$ is an eienvalue of $A$.
%
%Therefore we have $\rho\left(Y^{-1/2}XY^{-1/2}\right)\le 1$,
%since $Y^{-1}X=Y^{-1/2}\left(Y^{-1/2}X\right)$ and $\left(Y^{-1/2}X\right)Y^{-1/2}$ have same eigenvalues. Thus we have $I-Y^{-1/2}XY^{-1/2}\ge 0$ which implies that $Y-X=Y^{1/2}\left(I-Y^{-1/2}XY^{-1/2}\right)Y^{1/2}\ge 0$.

%(iii) From the condition  $X\le Y$, we have $\lambda_1(X)\le \lambda_1(Y)$ which implies $\lambda_1(X)^s\le \lambda_1(Y)^s$. Thus we have $\lambda_1(X^s)\le \lambda_1(Y^s)$ since $\lambda_1(A^s)=\lambda_1(A)^s$ for $s\ge 0$ and $A>0$ in general.
\end{proof}

\begin{theorem}\label{theorem3.1} 
Let $-\frac12\le t \le 0$ or $1\leq t \leq \frac{3}{2}$. For positive definite matrices $A$ and $B$, 
\begin{equation} \label{eq:m12}
A \hat{\spe}_t B \prec_{\log} A^{\frac{1-t}{2}}B^{t}A^{\frac{1-t}{2}}.
\end{equation}
\end{theorem}
\begin{proof}
Let $-\frac{1}{2}\leq t \leq 0$. We consider the Grassmann product and the $k$--th compound matrix $C_k(A)$ \cite[Chapter 19]{MO1979} for $n\times n$ positive definite matrix $A$
and the use of the known fact  $\prod\limits_{i=1}^k\lambda_i(A)=\lambda_1\left(C_k(A)\right)$ for $k=1,2,\cdots,n$, where $\lambda_i(A)$ represents the $i$--th largest eigenvalue of $A$. Since we have $\det A \hat{\spe}_t B=(\det A)^{1-t} (\det B)^t = \det \left(A^{\frac{1-t}{2}}B^{t}A^{\frac{1-t}{2}}\right)$, we have only to prove
$$
\lambda_1\left(C_k\left( A \hat{\spe}_t B\right)\right) \le \lambda_1\left(C_k\left(A^{\frac{1-t}{2}}B^{t}A^{\frac{1-t}{2}}\right)\right),\,\,(-\frac12\le t \le 0,\,\,k=1,2,\cdots,n-1).
$$
For $k=1,2,\cdots,n$ and $n\times n$ matrices $A,B$, we have the properties such as
$C_k(AB)=C_k(A)C_k(B)$, $C_k(A^*)=C_k(A)^*$, and $C_k(A^{-1})=C_k(A)^{-1}$ if $A$ is non--singular.
Thus we have only to prove  for $-\frac12\le t \le 0$,
\begin{equation}\label{theorem3.1_proof_eq01}
\lambda_1\left( A \hat{\spe}_t B\right) \le \lambda_1\left(A^{\frac{1-t}{2}}B^{t}A^{\frac{1-t}{2}}\right). %,\,\,(k=1,2,\cdots,n-1).
\end{equation}
In addition, for $\alpha, \beta >0$ we have  homogeneity as 
$$\left(\alpha A\right)\hat{\spe}_t \left(\beta B\right) =\alpha^{1-t}\beta^{t}(A \hat{\spe}_t B) \quad \text{and}\quad \left(\beta B\right)^{t/2}\left(\alpha A\right)^{1-t}\left(\beta B\right)^{t/2}=\alpha^{1-t}\beta^{t}\left(A^{\frac{1-t}{2}}B^{t}A^{\frac{1-t}{2}}\right).$$
Therefore it is sufficient to prove  for $-\frac12\le t \le 0$,
\begin{equation}\label{theorem3.1_proof_eq02}
A^{\frac{1-t}{2}}B^{t}A^{\frac{1-t}{2}}\le I \quad {\rm implies}\quad A \hat{\spe}_t B\le I,
\end{equation}
in order to prove \eqref{theorem3.1_proof_eq01}.

Assume $A^{\frac{1-t}{2}}B^{t}A^{\frac{1-t}{2}}\le I$ which implies $A\le B^{\frac{-t}{1-t}}$ since $0< \frac{1}{1-t} \le 1$ for all $t\le 0$. Thus we have $B^{\frac{1}{2(1-t)}}AB^{\frac{1}{2(1-t)}}\le B=CAC$ which is  the Riccati equation by  setting $C:=A^{-1}\sharp B$. By Lemma \ref{lemma3.1} (ii), we have
$B^{\frac{1}{2(1-t)}}\le C$ so that  $C^{2t} \le B^{\frac{t}{1-t}}$ for $-1\le 2t\le 0$ by Lemma \ref{lemma3.1} (i). This with $A\le B^{\frac{-t}{1-t}}$ , we have $A\le B^{\frac{-t}{1-t}}\le C^{-2t}$.
Thus we have $C^tAC^t \le I$ so that we obtained the implication \eqref{theorem3.1_proof_eq02} for $-\frac12\le t \le 0$.

Let $1\leq t \leq \frac{3}{2}$. Since $-\frac{1}{2}\leq 1-t \leq 0$, it follows from \eqref{theorem3.1_proof_eq01} that
\[
\lambda_1\left( B \hat{\spe}_{1-t} A\right) \le \lambda_1\left(B^{\frac{t}{2}}A^{1-t}B^{\frac{t}{2}}\right)
\]
and hence
\[
\lambda_1\left( A \hat{\spe}_t B\right) \le \lambda_1\left(A^{\frac{1-t}{2}}B^{t}A^{\frac{1-t}{2}}\right). 
\]
Therefore we have the desired inequality \eqref{eq:m12} for $1\leq t \leq \frac{3}{2}$.

%We prove the implication \eqref{theorem3.1_proof_eq02} for $1\le t \le \frac{3}{2}$.
%Replacing $A$ by $B$ and setting $s:=1-t$ in \eqref{theorem3.1_proof_eq02}, we have $-\frac{1}{2}\le s \le 0$ from $1\le t \le \frac{3}{2}$ and 
%\begin{align*}
%A^{\frac{1-t}{2}}B^{t}A^{\frac{1-t}{2}}\le I &\Longleftrightarrow A^{\frac{1-s}{2}}B^sA^{\frac{1-s}{2}}\le I \\
%& \Longrightarrow C^sAC^s \le I \quad ({\rm by \,\,\eqref{theorem3.1_proof_eq02}},\,\, ({\rm where}\,\,\, C:=A^{-1}\s B) \\
%& \Longleftrightarrow \tilde{C}^{1-t}B\tilde{C}^{1-t}\le I\,\,\,({\rm where}\,\,\,  \tilde{C}:=B^{-1}\s A=A\s B^{-1})\\
%& \text{\qquad(we replaced $A$ by $B$ again and set $s:=1-t$)}\\
%& \Longleftrightarrow C^{t-1}BC^{t-1}\le I\,\,\, ({\rm since}\,\,\, \tilde{C}^{-1}=C)\\
%& \Longleftrightarrow C^t AC^t \le I\,\,\, ({\rm since}\,\,\, B=CAC).
%\end{align*}
%Thus we obtained the implication \eqref{theorem3.1_proof_eq02} for $1\le t \le \frac{3}{2}$.\
\end{proof}

\begin{remark}
\begin{itemize}
\item[(i)] Theorem \ref{theorem3.1} is a direct consequence of Theorem \ref{norm_main} with the homogeneity and the determinant of $A\hat{\spe}_tB$ and $A^{\frac{1-t}{2}}B^tA^{\frac{1-t}{2}}$.
Since the proof in  Theorem \ref{theorem3.1} is not trivial for the readers, we gave its proof.
\item[(ii)]Since we have $\lambda_1(A^s)=\lambda_1(A)^s$ for $s\ge 0$ and $A>0$, we have the following relation under the same assumption in Theorem \ref{theorem3.1}:
\begin{equation*}
\left(A^s \hat{\spe}_t B^s\right)^{1/s} \prec_{\log} \left(A^{\frac{(1-t)s}{2}}B^{ts}A^{\frac{(1-t)s}{2}}\right)^{1/s},\quad (s>0)
\end{equation*}
by the similar way to the proof in Theorem \ref{theorem3.1}.
\end{itemize}
\end{remark}

We do not know whether the log-majorization \eqref{eq:m12} holds or not for $t\in [-1,-\frac{1}{2}]$.  However, in the case of $t=-1$, we have the following theorem.
To prove it, we need the following reverse BLP inequality \cite{MF2009,FMPS2}:
\begin{lemma}{{\rm (R-BLP inequality)}}\label{thm-RBLP}\ 
Let $A$ and $B$ be positive definite matrices. Then
\[
\NORM{A^{\frac{1+r}{2}} B^r A^{\frac{1+r}{2}}} \geq \NORM{A^{\frac{1}{2}}(A^{\frac{s}{2}} B^s A^{\frac{s}{2}})^{\frac{r}{s}} A^{\frac{1}{2}}}
\]
for all $r\geq s\geq 1$.
\end{lemma}

\begin{theorem} %\label{thm-m1}
Let $A$ and $B$ be positive definite matrices. Then %If $t =-1$ in \eqref{eq:m12}, then
\[
A \hat \spe_{-1}  B \prec_{\log} AB^{-1}A.
\]
\end{theorem}

\begin{proof}
It follows from R-BLP inequality that 
\begin{align*}
\NORM{A \hat \spe_{-1}  B} & = \NORM{A^{1/2}(A^{-1} \s B)^{-2} A^{1/2}} \\
& = \NORM{A^{1/2}(A \s B^{-1})^2 A^{1/2}} \\
& = \NORM{A^{1/2} ( A^{1/2}(A^{-1/2}B^{-1}A^{-1/2})^{1/2} A^{1/2})^2 A^{1/2}}\\
& \leq \NORM{A^{1/2}(A(A^{-1/2}B^{-1}A^{-1/2})A)A^{1/2}}\quad \mbox{(by $r=2, s=1$ in Lemma~\ref{thm-RBLP})}\\
& = \NORM{AB^{-1}A}.
\end{align*}
Hence $AB^{-1}A\leq I$ implies $A \hat \spe_{-1}  B \leq I$, and so we have the desired log-majorization.
\end{proof}

Generalizing the above method a bit more, we obtain the following result:

\begin{theorem}\label{theorem4.7}
Let $A$ and $B$ be positive definite matrices. If $-1\leq t \leq -\frac{1}{2}$, then
\[
 A \hat \spe_t B \prec_{\log} (A^{1-t}B^{2t}A^{1-t})^{\frac{1}{2}} = |A^{1-t} B^{t}|.
\]
\end{theorem}
\begin{proof}
It suffices to prove the following norm inequality:
\begin{align*}
\NORM{A \hat \spe_t B} & = \NORM{A^{1/2} (A\ \s\   B^{-1})^{-2t} A^{1/2}} \\
& \leq \NORM{A^{1/2}A^{-t}(A^{1/2}BA^{1/2})^{t}A^{-t}A^{1/2}} \quad \mbox{(by $r=-2t, s=1$ in Lemma~\ref{thm-RBLP})}\\
& = \NORM{A^{\frac{1}{2}-t}(A^{-1/2}B^{-1}A^{-1/2})^{-t}A^{\frac{1}{2}-t}} \\
& \leq \NORM{A^{\frac{2t-1}{2t}}(A^{-1/2}B^{-1}A^{-1/2})A^{\frac{2t-1}{2t}}}^{-t} \\
& = \NORM{A^{\frac{t-1}{2t}}B^{-1}A^{\frac{t-1}{2t}}}^{-t} \quad \mbox{(by $\frac{1}{2}\leq -t \leq 1$)}\\
& \leq\NORM{A^{1-t}B^{2t}A^{1-t}}^{\frac{1}{2}}\quad \mbox{(by $\frac{1}{2}\leq -\frac{1}{2t} \leq 1$)}.
\end{align*}
\end{proof}

$Q_{a,z}(A,B):=\left(A^{\frac{1-a}{2z}}B^{\frac{a}{z}}A^{\frac{1-a}{2z}}\right)^z$ for $a,z >0$ is often called R\'enyi mean \cite{DF2024}.  In the following, we consider the case $a:=t$ and $z:=t$ such that $Q_{t,t}(A,B)=\left(A^{\frac{1-t}{2t}}BA^{\frac{1-t}{2t}}\right)^t$ for $t\in \mathbb{R}$ as  R\'enyi type mean.
From \cite[Corollary 8]{DF2024}, we have 
$$
\left(A^{\frac{1-t}{2t}}BA^{\frac{1-t}{2t}}\right)^t\prec_{\log}  A \spe_t B,\quad \left(\frac{1}{2}\le t \le 1\right)
$$
for positive definite matrices $A$ and $B$. The following result gives us a counterpart $\left(0<t\le 1/2\right)$ and a negative part $\left(-1\le t <0\right)$ for the above.

\begin{theorem}\label{cor4.8}
Let $A$ and $B$ be positive definite matrices. If $-1\leq t \le \frac12$ with $t\neq 0$, then
$$
A \hat \spe_t B \prec_{\log}  \left(A^{\frac{1-t}{2t}}BA^{\frac{1-t}{2t}}\right)^t.
$$
\end{theorem}
\begin{proof}
From the process of the proof in Theorem \ref{theorem4.7} and $\NORM{X^p}=\NORM{X}^p$ for $p>0$ and $X>0$, it holds
\begin{equation}\label{cor4.8_eq01}
 A \hat \spe_t B \prec_{\log}  \left(A^{\frac{1-t}{2t}}BA^{\frac{1-t}{2t}}\right)^t,\quad \left(-1\leq t \leq -\frac{1}{2}\right).
 \end{equation}
 The following inequalities are known \cite[Theorem IX.2.10]{B1996}:
\begin{equation}\label{cor4.8_eq00}
\NORM{ X^p Y^p X^p } \le \NORM{ (XYX)^p }, \quad (0\le p\le 1)
 \end{equation}
 and
 \begin{equation}\label{cor4.8_eq001}
 \NORM{ X^p Y^p X^p } \ge \NORM{ (XYX)^p }, \quad (p\ge1)
  \end{equation}
for positive definite matrices $X,Y$. From this, we get
$$
\NORM{ \left(A^{\frac{t-1}{2t}}\right)^{-t} \left(B^{-1}\right)^{-t} \left(A^{\frac{t-1}{2t}}\right)^{-t}}
\le \NORM{\left(A^{\frac{t-1}{2t}}B^{-1}A^{\frac{t-1}{2t}}\right)^{-t}},\quad (-1\le t <0)
$$
which is equivalent to
$$
\NORM{A^{\frac{1-t}{2}}B^tA^{\frac{1-t}{2}}}\le 
\NORM{\left(A^{\frac{1-t}{2t}}BA^{\frac{1-t}{2t}}\right)^t},\quad (-1\le t <0)
$$
which implies 
$$
A^{\frac{1-t}{2}}B^tA^{\frac{1-t}{2}} \prec_{\log} \left(A^{\frac{1-t}{2t}}BA^{\frac{1-t}{2t}}\right)^t,\quad (-1\le t <0)
$$
since $A^{\frac{1-t}{2}}B^tA^{\frac{1-t}{2}}$ and $\left(A^{\frac{1-t}{2t}}BA^{\frac{1-t}{2t}}\right)^t$ are homogeneous and they have same determinants.
%$\det\left(A^{\frac{1-t}{2}}B^tA^{\frac{1-t}{2}}\right)=\det\left(\left(A^{\frac{1-t}{2t}}BA^{\frac{1-t}{2t}}\right)^t\right)$.
This with Theorem \ref{theorem3.1}, we obtain
\begin{equation}\label{cor4.8_eq02}
A \hat{\spe}_t B \prec_{\log} A^{\frac{1-t}{2}}B^{t}A^{\frac{1-t}{2}} \prec_{\log} \left(A^{\frac{1-t}{2t}}BA^{\frac{1-t}{2t}}\right)^t,\quad \left(-\frac12\le t <0\right).
 \end{equation}
For $0< t \le \dfrac12$,
\begin{align*}
\NORM{A {\spe}_t B} & = \NORM{A^{\frac{1}{2}}(A^{-1}\ \s\ B)^{2t}A^{\frac{1}{2}}}\\
& \leq \NORM{\left\{A^{\frac{1}{4t}}(A^{-1}\ \s\ B)A^{\frac{1}{4t}}\right\}^{2t}}\qquad \left({\rm by}\,\,\,\eqref{cor4.8_eq00}\,\,\,{\rm and}\,\,\, 0\leq 2t\leq 1\right)\\
& =\NORM{\left\{A^{\frac{1-2t}{4t}}(A^{\frac{1}{2}}BA^{\frac{1}{2}})^{\frac{1}{2}}A^{\frac{1-2t}{4t}}\right\}^{2t}}\\
& \leq \NORM{(A^{\frac{1-t}{2t}}BA^{\frac{1-t}{2t}})^t}\qquad \left({\rm by}\,\,\,\eqref{cor4.8_eq001}\,\,\,{\rm with} \,\,p:=2\right)\\
\end{align*}
which shows
$$
A \spe_t B \prec_{\log} \left(A^{\frac{1-t}{2t}}BA^{\frac{1-t}{2t}}\right)^t,\quad \left(0< t \le \frac12\right).
$$
 From this with \eqref{cor4.8_eq01} and \eqref{cor4.8_eq02}, the proof is complete. 
\end{proof}
%\begin{remark}
%From the process of the proof in Theorem \ref{theorem4.7} and $\NORM{X^p}=\NORM{X}^p$ for $p>0$ and $X>0$, it holds
%$$
% A \hat \spe_t B \prec_{\log}  \left(A^{\frac{1-t}{2t}}BA^{\frac{1-t}{2t}}\right)^t,\quad \left(-1\leq t \leq -\frac{1}{2}\right).
%$$
%On the other hand, we have
%$$
%\left(A^{\frac{1-t}{2t}}BA^{\frac{1-t}{2t}}\right)^t\prec_{\log}  A \spe_t B,\quad \left(\frac{1}{2}\le t \le 1\right)
%$$
%from \cite[Corollary 8]{DF2024}.
%\end{remark}

\begin{remark}
By the similar way to the proof in Theorem \ref{cor4.8}, we have
$$
\NORM{ \left(A^{\frac{a-1}{2z}}\right)^{-z} \left(B^{-\frac{a}{z}}\right)^{-z} \left(A^{\frac{a-1}{2z}}\right)^{-z}}
\le \NORM{\left(A^{\frac{a-1}{2z}}B^{-\frac{a}{z}}A^{\frac{a-1}{2z}}\right)^{-z}},\quad (-1\le z \le\dfrac12,\,z\neq 0,\,\,a\in\mathbb{R})
$$
for positive definite matrices $A$ and $B$. This is equivalent to the inequality
$$
\NORM{ A^{\frac{1-a}{2}}B^{a}A^{\frac{1-a}{2}}}
\le \NORM{\left(A^{\frac{1-a}{2z}}B^{\frac{a}{z}}A^{\frac{1-a}{2z}}\right)^{z}},\quad (-1\le z \le\dfrac12,\,z\neq 0,\,\,a\in\mathbb{R})
$$
which implies 
$$
A^{\frac{1-a}{2}}B^{a}A^{\frac{1-a}{2}} \prec_{\log} Q_{a,z}(A,B),\quad (-1\le z \le\frac12,\,\,z\neq 0, \,\,\,a\in\mathbb{R}).
$$
From Theorem \ref{theorem3.1}, we have
$$
A\hat\spe_a B \prec_{\log} Q_{a,z}(A,B),\quad \left(z\in [-1,0)\cup(0,1/2],\,\,a \in [-1/2,0] \cup [1,3/2]\right).
$$
\end{remark}

\par
\medskip
\medskip

Finally, we show the relation among three quantum Tsallis relative entropies: 

\begin{corollary}\label{corollary3.2}
Let  $A$ and $B$ be positive definite matrices  and $t\neq 0$.
If $-\frac12\le t \le 1$, then
\begin{equation}\label{corollary3.2_eq01}
-\tr\,T^{\spe}_t(A|B)\le D_t(A|B) \le -\tr\, T_t(A|B).
\end{equation}
If $1\le t \le \frac32$, then the revrsed inequalities for \eqref{corollary3.2_eq01} hold.
\end{corollary}
\begin{proof}
For the case $0<t\le 1$, it is known in \eqref{sec3_eq01}. From Theorem \ref{theorem3.1}, we have
$\tr\,A\hat \spe_t B \le \tr\,A^{1-t}B^t$ for $-\frac12\le t \le 0$. % under the assumption $A\le I$. 
Also the trace inequality  $\tr\,A^{1-t}B^t \le \tr\,A\hat{\sharp}_t B$ for $-1\le t <0$ was shown in \cite[Theorem 2.2.]{S2019}. Thus we have the inequalities \eqref{corollary3.2_eq01} for $-\frac12\le t <0$.

If $1\le t \le \frac32$, then $\tr A\hat \spe_t B \le \tr A^{1-t}B^t \le \tr A\hat \sharp_t B$ by Proposition \ref{prop_known_fact} (ii) and Theorem \ref{theorem3.1}. Thus we get the desired result.
\end{proof}

\begin{remark}
For any matrix $X$ and positive definite matrices $A,B$, the quasi--entropy (quantum $f$--divergence) \cite{HP2014,P1986} was introduced by
$$
S_f^X(A|B)=\tr\, X^* \mathcal P_f(L_A,R_B)X,\qquad \mathcal P_f(L_A,R_B):=f\left(L_A R_B^{-1}\right)R_B,
$$
where the perspective $\mathcal P_f(A,B):=B^{1/2}f\left(B^{-1/2}AB^{-1/2}\right)B^{1/2}$ is defined for $f:(0,\infty)\to\mathbb{R}$. Here two multiplication operators $L_A(X):=AX$ and $R_B(X):=XB$ are commutative.
The maximal $f$--divergence \cite{HM2017,PR1998} was also introduced as
$$
\hat{S}_f(A|B):=\tr Bf\left(B^{-1/2}AB^{-1/2}\right). 
$$
If $f:(0,\infty)\to \mathbb{R}$ is operator convex, then the following inequality is known \cite{M2013,HM2017}:
\begin{equation}\label{MHM_ineq}
S_f^I(A|B)\le \hat{S}_f(A|B).
\end{equation}

Since the function $x^{1-t}\ln_tx=\dfrac{x-x^{1-t}}{t}$ for $-1\le t \le 1,\,\,(t\neq 0)$  is operator convex, we obtain the second inequality in \eqref{corollary3.2_eq01} as a special case of the inequality \eqref{MHM_ineq}. However, to the best of our knowledge, we can not obtain the first inequality from the known results.
\end{remark}

Taking the limit $t\to 0$ in Corollary \ref{corollary3.2}, we have the following corollary.
\begin{corollary}\label{corollary4.7}
For positive definite matrices $A$ and $B$, 
\begin{equation*}%\label{bounds_relative_entropy}
-\tr\,A\log \left(A^{-1}\sharp B\right)^2 \le S(A|B) \le -\tr\, A\log \left(A^{-1/2}BA^{-1/2}\right),
\end{equation*}
where $S(A|B):=\tr\, \left(A\log A-A\log B\right)$ is the Umegaki relative entropy \cite{Umegaki}. 
\end{corollary}

The above inequalities give us the upper and lower bounds for the Umegaki relative entropy $S(A|B)$. If $A$ and $B$ are commutative, then the both sides are equal to the relative entropy.

\section{Concluding remarks}
We have studied the mathematical properties and inequalities for the weighted spectral geometric mean. As an application, we give a new lower bound of the Tsallis relative entropy by defining the Tsallis relative operator entropy due to the weighted spectral geometric mean.

The upper bound $-\tr\, A\log \left(A^{-1/2}BA^{-1/2}\right)$ in Corollary \ref{corollary4.7} is often called as the Belavkin--Staszewski \cite{BS1982} or the Fujii-Kamei relative entropy \cite{FK1989} for positive definite matrices $A$ and $B$. The second inequality  in Corollary \ref{corollary4.7}  is known in \cite{HP1991,HP1993}.
On the other hand, the lower bound in Corollary \ref{corollary4.7} coincides to the Ando--Hiai inequality \cite[Theorem 5.2]{AH1994} with $p=1$ and $B:=B^{-1}$:
$$
\tr A(\log A+\log B) \ge \frac{1}{p}\tr A\log(A^p\sharp B^p)^2,\,\,(A,B\ge 0,\,\,p>0).
$$

Closing this paper, we give discussions on the lower bound of the Umegaki relative entropy.
It is known that the Umegaki relative entropy has a lower bound, which is so--called Pinsker inequality, \cite[Theorem 9.17]{AF2001} and \cite[Example 3.4]{Petz2008} such as
$$
S(A|B) \ge \frac{1}{2} \tr \left[|A-B|\right]^2 
$$
for density matrices $A$ and $B$, where $|X|:=(X^*X)^{1/2}$. 
It is also known \cite{HP1993} that the Umegaki relative entropy has an another lower bound
$$S(A|B) \ge \tr A\log B^{-1/2}AB^{-1/2}$$
for positive definite matrices $A,B$.
So we may have an interest in the ordering among $\frac{1}{2} \tr \left[|A-B|\right]^2$, $\tr A\log B^{-1/2}AB^{-1/2}$ and $\tr\,A\log \left(A\sharp B^{-1}\right)^2$ for density matrices $A,B$. Then we have the following numerical examples.

%It is natural to compare two lower bounds of the relative entropy.
%We give a simple example such that 
%\begin{equation}\label{lower_bounds_relative_entropy}
%S(A|B) \ge \tr\,A\log \left(A\sharp B^{-1}\right)^2 \ge \frac{1}{2} \tr \left[|A-B|\right]^2.
%\end{equation}
%Take 
%\[A: =\frac17 \left( {\begin{array}{*{20}{c}}
%2&2\\
%2&5
%\end{array}} \right),\,\,B: = \frac14\left( {\begin{array}{*{20}{c}}
%2&{ - 1}\\
%{ - 1}&2
%\end{array}} \right).\]
%Then we have
%$$
%S(A|B)\simeq 0.740761,\quad \tr\,A\log \left(A\sharp B^{-1}\right)^2 \simeq 0.728114,\quad
% \frac{1}{2} \tr \left[|A-B|\right]^2 \simeq 0.665816.
%$$
%This numerical example supports that our lower bound $-\tr\,A\log \left(A^{-1}\sharp B\right)^2$ of $S(A|B)$ is better than the known one $\frac{1}{2} \tr \left[|A-B|\right]^2$.
%In addition, the equalities in \eqref{bounds_relative_entropy} hold when $A$ and $B$ are commutative.
%This also supports the advantage for our lower bound.
%\par
%\medskip

\begin{enumerate}
\item[(i)] For two density matrices $A: =\dfrac17 \left( {\begin{array}{*{20}{c}}
2&2\\
2&5
\end{array}} \right)\,\,{\rm and}\,\,\,B: = \dfrac14\left( {\begin{array}{*{20}{c}}
2&{ 1}\\
{ 1}&2
\end{array}} \right)$,
$$\tr\,A\log \left(A\sharp B^{-1}\right)^2 \simeq 0.10285 \ge \tr A\log B^{-1/2}AB^{-1/2} \simeq 0.0984248 \ge \frac{1}{2} \tr \left[|A-B|\right]^2 \simeq 0.0943878.$$

\item[(ii)] For two density matrices $A: =\dfrac17 \left( {\begin{array}{*{20}{c}}
2&2\\
2&5
\end{array}} \right)\,\,{\rm and}\,\,\,B: = \dfrac14\left( {\begin{array}{*{20}{c}}
2&{ -1}\\
{ -1}&2
\end{array}} \right)$,
$$\tr A\log B^{-1/2}AB^{-1/2}{\simeq 0.728623}  \ge \tr\,A\log \left(A\sharp B^{-1}\right)^2{\simeq 0.728114} \ge \frac{1}{2} \tr \left[|A-B|\right]^2{\simeq 0.665816}.$$

\item[(iii)] For two density matrices $A: =\dfrac17 \left( {\begin{array}{*{20}{c}}
2&2\\
2&5
\end{array}} \right)\,\,{\rm and}\,\,\,B: = \dfrac{1}{10}\left( {\begin{array}{*{20}{c}}
1&{ 0}\\
{ 0}&9
\end{array}} \right)$,
$$\tr\,A\log \left(A\sharp B^{-1}\right)^2{\simeq 0.260302}\ge \frac{1}{2} \tr \left[|A-B|\right]^2{\simeq 0.232245} \ge \tr A\log B^{-1/2}AB^{-1/2}{\simeq 0.227952}.$$
\end{enumerate}

From the above numerical computations, we may state that there is no ordering among three bounds.
However, we have not found an example such that $\dfrac{1}{2} \tr \left[|A-B|\right]^2 \ge \tr\,A\log \left(A\sharp B^{-1}\right)^2$. \par
\medskip

\textbf{Acknowledgements.}  
The authors would like to thank the referees for their careful and insightful comments to improve our manuscript.
The authors are partially supported by JSPS KAKENHI Grant Number JP21K03341 and JP23K03249, respectively.

\par


\medskip




 \begin{thebibliography}{99} 
\bibitem{AF2001} R. Alicki and M. Fannes, {\it Quantum dynamical systems}, Oxford University Press, 2001.

 \bibitem{A1979} T. Ando, {\it Concavity of certain maps on positive definite matrices and applications to Hadamard product}, Linear Algebra Appl., {\bf 26} (1979), 203--241.

\bibitem{AH1994} T. Ando and F. Hiai, {\it Log majorization and complementary Golden--Thompson type inequalities}, Linear Algebra Appl.,{\bf 197/198}(1994), 113--131.

 \bibitem{BCP2014} T. Baumgratz, M. Cramer and M. B. Plenio, {\it Quantifying coherence}, Phys. Rev. Lett., {\bf 113}(14)(2014), 140401.

\bibitem{BLP2005} N. Bebiano, R. Lemos, J. da Provid\^{e}ncia, {\it Inequalities for quantum relative entropy}, Linear Algebra Appl., {\bf 401} (2005),  159--172.
 
 \bibitem{B1996} R. Bhatia, {\it Matrix Analysis}, Springer--Verlag, 1996.
 
%\bibitem{BG} R. Bhatia and P. Grover, {\it Norm inequalities related to the matrix geometric mean}, Linear Algebra Appl., {\bf 437} (2012), 726--733.

 \bibitem{BS1982} 
V. P. Belavkin and P. Staszewski, {\it $C^{*}$-algebraic generalization of relative entropy and entropy}, Ann. Inst. H. Poincar\'e Sect. A, {\bf 37} (1982), 51--58.
 
 \bibitem{BLMS2009} J. C. Bourin,  E. Y. Lee, M. Fujii, and Y. Seo, {\it A matrix reverse H\"older inequality},  Linear Algebra Appl., {\bf 431} (11) (2009),  2154--2159.

\bibitem{DF2024} R. Dumitru and J. A. Franco, {\it Log--majorization related to the spectral geometric and R\'enyi means}, Acta Sci. Math. (Szeged), \url{https://doi.org/10.1007/s44146-024-00128-8}.

\bibitem{FK1989}
J. I. Fujii and E. Kamei, {\it Relative operator entropy in non--commutative information theory}, Math. Japon., {\bf34} (1989), 341--348.

\bibitem{FMPS2} 
M. Fujii, J. Mi\'{c}i\'{c} Hot, J. Pe\v{c}ari\'{c} and Y. Seo, 
{\it Recent Developments of Mond--Pe\v{c}ari\'{c} Method in Operator Inequalities}, 
Monographs in Inequalities {\bf 4}, Element, Zagreb, 2012. 
 
 \bibitem{FS2021} M. Fujii and Y. Seo, {\it Matrix trace inequalities related to the Tsallis relative entropies of real order}, J. Math. Anal. Appl., {\bf 498} (2021), 124877.
 
 \bibitem{FYK2004} S. Furuichi, K. Yanagi and K. Kuriyama, Fundamental properties of Tsallis relative entropy, J. Math. Phys., {\bf 45} (2004), 4868--4877.
 
 \bibitem{FM2020book} S. Furuichi and H. R. Moradi, {\it Advances in mathematical inequalities}, De Gruyter, 2020.
% \bibitem{GLT2021} L.Gan, X.Liu and T.-Y. Tam, {\it On two geometric means and sum of adjoint orbits},  Linear Algebra Appl., {\bf 631}(2021), 156--173.
 
\bibitem{FMPS1} 
T. Furuta, J. Mi\'{c}i\'{c} Hot, J. Pe\v{c}ari\'{c} and Y. Seo, 
{\it Mond-Pe\v{c}ari\'{c} Method in Operator Inequalities}, 
Monographs in Inequalities {\bf 1}, Element, Zagreb, 2005. 

\bibitem{GT2022} L. Gan and T.-Y. Tam, {\it Inequalities and limits of weighted spectral geometric mean}, Linear and Multilinear Algebra, {\bf 72}(2)(2024), 261--282.
  
\bibitem{GK2024} L. Gan and S. Kim, {\it Weak log--majorization between the geometric and Wasserstein means}, J. Math. Anal. Appl., {\bf 530} (2024), 127711.
  
  
\bibitem{GJBBHF2020}  M.--L. Guo, Z.--X. Jin, B. Li, B. Hu, and S.--M. Fei, {\it Quantifying quantum coherence based on the Tsallis relative operator entropy},  Quantum Information Processing, 
 {\bf 19} (2020), Art. 382.

 \bibitem{H2019} F. Hiai, {\it Log--majorization related to Renyi divergence}, Linear Algebra Appl., {\bf 563} (2019), 255--276.

\bibitem{HM2017} F. Hiai and M. Mosonyi, {\it Different quantum $f$--divergences
and the reversibility of quantum operations}, Rev. Math. Phys., {\bf 29} (2017), 1750023.


\bibitem{HP1991} F. Hiai and D. Petz,  {\it The proper formula for relative entropy in asymptotics in quantum probability}, Commun. Math. Phys., {\bf 143}(1991), 99--114.

\bibitem{HP1993} F. Hiai and D. Petz, {\it The Golden--Thompson trace inequality is complemented}, Linear Algebra Appl., {\bf 181}(1993), 153--185.

\bibitem{HP2014} F. Hiai and D. Petz, {\it Introduction to matrix analysis and applications}, Springer--Verlag, 2014.
 
\bibitem{K1948} L. V. Kantorovich, {\it Functional analysis and applied mathematics} (in Russian), Uspekhi Mat. Nauk. {\bf 3} (6) (1948), 89--185.
 
 
 \bibitem{K2021} S. Kim, {\it Operator inequalities and gyrolines of the weighted geometric means},
 Math. Inequal. Appl., {\bf 24}(2)(2021), 491--514.
 
 \bibitem{KL} S. Kim and H. Lee, {\it Relative operator entropy related with the spectral geometric mean}, Anal. Math. Phys., {\bf 5}(3)(2015), 233--240.
 
\bibitem{KA} F. Kubo and T. Ando, {\it Means of positive linear operators}, Math. Ann., {\bf 246} (1980), 205--224.

 \bibitem{L2009} E. Y. Lee,  {\it A matrix reverse Cauchy--Schwarz inequality},  Linear Algebra Appl., {\bf 430} (2/3) (2009), 805--810.
 
  \bibitem{LL2007} H. Lee and Y. Lim, {\it Metric and spectral geometric means on symmetric cones}, Kyungpook Math. J., {\bf47} (1) (2007), 133--150.
 
% \bibitem{M1970} F. T. Man, {\it Some inequalities for positive definite symmetric matrices}, SIAM J. Appl. Math., {\bf 19} (4) (1970), 679--681,
 
 \bibitem{MO1979}  
A. W. Marshall and I. Olkin, {\it Inequalities: Theory of majorization and its applications}, Academic Press, 1979.
 
 \bibitem{MF2009} A. Matsumoto and M. Fujii, {\it Generalizations of reverse Bebiano--Lemos--Provid\^{e}ncia inequality}, Linear Algebra Appl., {\bf 430} (5/6) (2009),  1544--1549.
 
\bibitem{M2013} K. Matsumoto, {\it A new quantum version of $f$--divergence}, arXiv:1311.4722v4.
 

 
  \bibitem{MFS2023} H. R. Moradi, S. Furuichi and M. Sababheh, {\it Operator spectral geometric versus geometric mean}, Linear and Multilinear Algebra, {\bf 72}(6)(2024), 997--1016.
  
\bibitem{NS} R. Nakamoto and Y. Seo, {\it A complement of the Ando-Hiai inequality and norm inequalities for the geometric mean}, Nihonkai Math. J., {\bf 18} (2007), 43--50.

\bibitem{P1986} D. Petz, {\it Quasi--entropies for finite quantum system}, Rep. Math. Phys., {\bf 23} (1986), 57--65.
 
 \bibitem{Petz2008} D. Petz, {\it Quantum information theory and quantum statistics}, Springer--Verlag, 2008.
 
  \bibitem{PR1998} D. Petz and M. B. Ruskai, {\it Contraction of generalized relative entropy under stochastic mappings on matrices} Infin. Dimens. Anal. Quant. Prob., {\bf 1}(1)(1998), 83--89,
1998.
 

 
 \bibitem{R2016} A. E. Rastegin, {\it Quantum--coherence quantifers based on the Tsallis relative $\alpha$ entropies}, Phys. Rev. A, {\bf 93} (2016), 032136.
 
\bibitem{S2007} Y. Seo, {\it Norm inequalities for the chaotically geomtric mean and its reverse}, J. Math. Inequal., {\bf 1} (2007), 39--43.

 \bibitem{S2019} Y. Seo, {\it Matrix trace inequalities on Tsallis relative entropy of negative order}, J. Math. Anal. Appl., {\bf 472} (2019), 1499--1508.
 
 \bibitem{Umegaki} 
H. Umegaki, {\it Conditional expectation in an operator algebra, IV (entropy and information)}, Kodai Math. Sem. Rep., {\bf 14} (1962), 59--85.
 
 \bibitem{ZY2018} H. Zhao and C.--S. Yu, {\it Coherence measure in terms of the Tsallis relative $\alpha$ entropy}, Scientific Reports, {\bf 8}  (2018), Art. 299. 
 \end{thebibliography}
\end{document}